\documentclass[10pt]{amsart}
\usepackage[utf8]{inputenc}

\usepackage{amssymb,amsthm,amsmath}
\usepackage{mathrsfs}
\usepackage{enumerate}
\usepackage{graphicx,xcolor}
\usepackage{setspace}
\usepackage[hidelinks]{hyperref}
\usepackage{mathtools}
\usepackage{tikz}

\newcommand{\dd}{\mathrm{d}}
\newcommand{\E}{\mathbb{E}}

\newcommand{\1}{\textbf{1}}
\newcommand{\R}{\mathbb{R}}

\newcommand{\p}[1]{\mathbb{P}\left( #1 \right)}
\newcommand{\scal}[2]{\left\langle #1, #2 \right\rangle}
\newcommand{\red}{}

\newcommand{\cE}{\mathcal{E}}

\DeclareMathOperator{\Var}{Var}

\DeclareMathOperator{\vol}{vol}

\newtheorem{theorem}{Theorem}
\newtheorem{lemma}[theorem]{Lemma}
\newtheorem{corollary}[theorem]{Corollary}

\theoremstyle{remark}
\newtheorem{remark}[theorem]{Remark}

\theoremstyle{definition}

\title{From simplex slicing to sharp reverse H\"older inequalities
}

\author{James Melbourne}

\address{(JM) Department of Probability and Statistics, Centro de Investigaci\'on en Matem\'aticas (CIMAT), Mexico.\vspace*{-1em}}

\author{Michael Roysdon}

\address{(MR) Department of Mathematics, Applied Mathematics, and Statistics, Case Western Reserve University, Cleveland, OH 44106, USA
\newline and
Department of Mathematical Sciences, University of Cincinnati, Cincinnati, OH 45221, USA.
\vspace*{-1em}
}

\author{Colin Tang}

\author{Tomasz Tkocz}

\address{(CT \& TT) Department of Mathematical Sciences, Carnegie Mellon University; Pittsburgh, PA 15213, USA.}

\email{ttkocz@math.cmu.edu}

\date{\today}

\begin{document}

\begin{abstract} 
Simplex slicing (Webb, 1996) is a sharp upper bound on the volume of central hyperplane sections of the regular simplex. We extend this to sharp bounds in the probabilistic framework of negative moments, and beyond, of centred log-concave random variables, establishing a curious phase transition of the extremising distribution for new sharp reverse H\"older-type inequalities.
\end{abstract}

\maketitle

\begin{center}
\begin{large}
{\red \emph{Dedicated to Keith Ball, on his 65$^{th}$ birthday.}}
\end{large}
\end{center}

\bigskip

\begin{footnotesize}

\end{footnotesize}

\bigskip

\section{Introduction}

\subsection{Motivation}

The problem of finding minimal and maximal volume sections --- \emph{critical sections} --- of various convex bodies has received significant attention over the last several decades. This topic, being somewhat tangential to the broad area of geometric tomography, has nevertheless played a pivotal role in its development, as witnessed for instance by a strikingly simple {\red counterexample} to the Busemann-Petty problem thanks to Ball's famed cube slicing result (see \cite{Ball-cube, Ball-BP} for details). 

Notably, the study of critical sections has also been fruitful in developing robust methods, with Fourier analytic ideas at the core (see \cite{Kol-book}), only recently to be {\red enhanced} by a nascent  probabilistic point of view involving negative moments of weighted sums of random variables (see, e.g. \cite{CKT, CNT, ENT3}). We refer to the recent survey \cite{NT-surv} showcasing this approach, as well as contextualising it further, with comprehensive references and historical account. This very line of thought serves as the main motivation for this paper.

Our starting point is Webb's simplex slicing result from \cite{W} and the question whether it admits a \emph{probabilistic extension} to negative moments, {\red akin to} the aforementioned Ball's cube slicing generalised in such a way in \cite{CKT}. 

Concretely, let us consider the regular $n$-dimensional simplex
\[ 
\Delta_n = \left\{x \in \R^{n+1}, \ \sum_{j=1}^{n+1} x_j = 1, \ x_1, \dots, x_{n+1} \geq 0\right\}
 \]
embedded in the hyperplane $\mathcal{H} = \{x \in \R^{n+1}, \ \sum x_j = 1\}$ in $\R^{n+1}$ (that is, the vertices of $\Delta_n$ are the standard basis vectors $e_1, \dots, e_{n+1}$).

Webb in \cite{W} showed that the maximal volume  \emph{central} section of the regular simplex is attained at hyperplanes containing all but two vertices of the simplex, namely
\[ 
\max_{{\red L}} \ \vol_{n-1}(\Delta_n \cap H) = \frac{\sqrt{n+1}}{(n-1)!}\cdot\frac{1}{\sqrt2},
\]
where the maximum is taken over all affine subspaces {\red $L$} of $\mathcal{H}$ of (relative) codimension $1$ constrained to pass through the barycentre $\frac{1}{n+1}(1, \dots, 1)$ of $\Delta_n$. Plainly, every such hyperplane {\red $L$} extends to a codimension $1$ subspace of $\R^{n+1}$ by taking the affine hull of {\red $L$} and the origin, yielding a hyperplane $a^\perp = \{x \in \R^{n+1}, \ \scal{a}{x} = 0\}$ in $\R^{n+1}$ with an outer-normal vector $a = (a_1, \dots, a_{n+1})$ satisfying $\sum_{j=1}^{n+1} a_j = 0$ (so that $a^\perp$ contains the barycentre of $\Delta_n$). (We use the standard Euclidean structure given by the inner product $\scal{x}{y} = \sum x_jy_j$ and resulting Euclidean norm $|x| = \sqrt{\scal{x}{x}}$ for arbitrary vectors $x$ and $y$.) With this identification and the normalisation $|a| = 1$, we have the following formula for volume of sections, instrumental in Webb's work,
\begin{equation}\label{eq:vol-formula}
\vol_{n-1}(\Delta_n \cap a^\perp) = \frac{\sqrt{n+1}}{(n-1)!}\cdot f_{\sum a_j\cE_j}(0).
 \end{equation}
For completeness, we include a standard derivation in the appendix. Here $f_{\sum a_j\cE_j}$ is (the continuous version of) the density of the random variable $\sum_{j=1}^{n+1} a_j\cE_j$, and $\cE_1, \dots, \cE_{n+1}$ are independent identically distributed (i.i.d.) standard exponential random variables (that is with density $e^{-x}\1_{(0,+\infty)}(x)$ on $\R$). Webb's result thus amounts to the following sharp upper bound on the density at $0$ of weighted sums of i.i.d. exponentials $\cE_j$,
\begin{equation}\label{eq:Webb}
f_{\sum a_j\cE_j}(0) \leq \frac{1}{\sqrt2},
 \end{equation}
for all unit vectors $a$ with $\sum a_j = 0$, with equality attained if and only if $a = \frac{e_j \pm e_k}{\sqrt{2}}$, for some $j \neq k$.

At the heart of the aforementioned negative moments paradigm lies the following elementary fact: for an integrable function $f$ on $\R$ which is, say, continuous at $0$, we have
\[ 
f(0) = \lim_{p \searrow -1} \frac{1+p}{2}\int_{\R} |x|^pf(x) \dd x,
 \]
for instance see Lemma 4 in \cite{HL} as well as Lemma 4.3 in \cite{CNT} for a multivariate version.
In view of this observation, Webb's result can be restated in {\red yet another} equivalent form, as the bound
\begin{equation}\label{eq:Webb-lim}
 \lim_{p \searrow -1} \frac{1+p}{2}\E\left|\sum_{j=1}^{n+1} a_j \cE_j\right|^p \leq \frac{1}{\sqrt2} =  \lim_{p \searrow -1} \frac{1+p}{2}\E\left|\frac{\cE_1-\cE_2}{\sqrt2}\right|^p.
\end{equation}
for all unit vectors $a$ with $\sum a_j = 0$. Does this bound continue to hold without taking the limit, say for all \emph{fixed} $p$ in some neighbourhood of $-1$?

\subsection{Main results.}
Recall that a real-valued random variable $X$ is called log-concave, if it is continuous with a log-concave density $f$ on $\R$, that is a function of the form $f = e^{-\phi}$ with a convex function $\phi\colon \R \to (-\infty, +\infty]$. For instance, a uniform random variable on an interval, or a Gaussian random variable is log-concave, and most importantly for this discussion, exponential random variables are log-concave. Moreover, it is well-known that convolutions of log-concave functions are log-concave (by the Pr\'ekopa-Leindler inequality), thus sums of independent log-concave random variables are log-concave. They naturally play a prominent role in modern convex geometry, see the monographs \cite{AGM, BGVV} for background. Needless to say, the weighted sums 
\[
X = \sum a_j\cE_j
\]
of independent exponential random variables $\cE_j$ have naturally {\red arisen} in simplex slicing, as just considered. These are log-concave random variables. Owing to the geometric constraints imposed on the weights $a_j$, they are centred, {\red that is,} with \emph{mean $0$},
\[ 
\E X = \sum a_j = 0,
 \]
and they have \emph{variance $1$},
\[ 
\Var(X) = \sum a_j^2 = 1.
 \]

Our main results address the question we left hanging. 

\begin{theorem}\label{thm:Lp-L2}
For every $-1 < p \leq 1$ and every mean $0$ log-concave random variable $X$, we have
\[ 
\|X\|_p \geq 2^{-1/2}\Gamma(p+1)^{1/p}\|X\|_2,
 \]
which is sharp with equality attained by a standard double-exponential random variable $X$ (with density $\frac{1}{2}e^{-|x|}$ on $\R$).
\end{theorem}

Here and throughout, we use the standard notation of $L_p$-norms\footnote{abusing the terminology slightly as these are not norms anymore when $p < 1$}:  for a random variable $X$ and $p \in \R$, 
\[
\|X\|_p = (\E |X|^p)^{1/p} \ \in [0,+\infty],
\]
with the usual convention of adopting the limiting expressions as definitions at $p = 0$ and $p = +\infty$, respectively as $\|X\|_0 = e^{\E\log|X|}$ (the geometric mean) and $\|X\|_\infty = \text{ess sup} |X|$ (the essential supremum). For log-concave random variables $X$, $\|X\|_p$ is finite for all $p \in (-1, +\infty)$ as a result of an exponential decay of their densities (see, e.g. Lemma 2.2.1 in \cite{BGVV}).

In particular, for any log-concave random variable $X$ with mean $0$ {\red and} variance $1$, and for every $-1 < p < 0$, Theorem \ref{thm:Lp-L2} yields
\[ 
\E|X|^p \leq 2^{-p/2}\Gamma(p+1) = 2^{-p/2}\frac{\Gamma(p+2)}{p+1}, 
 \]
so after taking the limit,
\[ 
f_X(0) = \lim_{p \searrow -1} \frac{1+p}{2}\E|X|^p \leq 2^{-1/2},
 \]
which, when specialized to sums of exponentials $X = \sum a_j\cE_j$ is Webb's result, see \eqref{eq:Webb} and \eqref{eq:Webb-lim}.

Theorem \ref{thm:Lp-L2} is in fact obtained as a corollary to  a sharp $L_p - L_1$ moment comparison inequality, where in lieu of the \emph{variance constraint}, we impose the \emph{$L_1$ constraint} (which goes hand in hand with the mean $0$ constraint, but more on that later). Intriguingly, there is a phase transition of the extremising distribution. Specifically, for $p \geq 1$, we define
\begin{equation}\label{eq:def-Cp}
C_p = \max\left\{\Gamma(p+1)^{1/p}, \ \frac{e}{2}\|\cE-1\|_p  \right\},
\end{equation}
where $\cE$ is a standard exponential random variable. In fact (see Lemma \ref{lm:Cp} in the appendix),
\begin{equation}\label{eq:Cp}
C_p =\begin{cases}\Gamma(p+1)^{1/p}, \ &1 \leq p \leq p_0, \\ 
\frac{e}{2}\|\cE-1\|_p, \ &p \geq p_0,
\end{cases}
\end{equation}
where $p_0 = 2.9414..$ is the unique solution to the equation $\Gamma(p+1)^{1/p} = \frac{e}{2}\|\cE-1\|_p$ in $(1, +\infty)$. Our main result reads as follows.

\begin{theorem}\label{thm:Lp-L1}
For every log-concave random variable $X$ with mean $0$ and every $-1 < p \leq 1$, we have
\begin{equation}\label{eq:Lp-L1-p<1}
\|X\|_p \geq \Gamma(p+1)^{1/p}\|X\|_1,
 \end{equation}
whilst for every $p \geq 1$, we have
\begin{equation}\label{eq:Lp-L1-p>1}
\|X\|_p \leq C_p\|X\|_1,
\end{equation}
with the constant $C_p$ from \eqref{eq:Cp}.
Both bounds are sharp with equalities attained at either double-exponential or one-sided exponential random variables.
\end{theorem}

Theorem \ref{thm:Lp-L2} follows at once. Indeed, using Theorem \ref{thm:Lp-L1} twice, for $-1 < p \leq 1$, we have
\[ 
\|X\|_p \geq \Gamma(p+1)^{1/p}\|X\|_1 \geq \Gamma(p+1)^{1/p}\frac{1}{C_2}\|X\|_2 =  2^{-1/2}\Gamma(p+1)^{1/p}\|X\|_2.
 \]
The same argument shows in fact the sharp $L_p - L_q$ moment comparison inequality in the range where the extremiser is double-exponential.

\begin{corollary}\label{cor:Lp-Lq}
For every log-concave random variable $X$ with mean $0$ and every $-1 < p \leq 1 \leq q \leq p_0$, we have
\begin{equation}\label{eq:Lp-Lq}
\|X\|_p \geq \frac{\Gamma(p+1)^{1/p}}{\Gamma(q+1)^{1/q}}\|X\|_q.
 \end{equation}
\end{corollary}

\subsection{Related works.}
Webb's simplex slicing result has recently stimulated several geometric results: a stability result for Webb's inequality \eqref{eq:Webb} has been established in \cite{MTTT}, and an asymptotically sharp reversal (i.e. a lower bound on volume) in \cite{Tang}, proceeding several partial results from \cite{Brz, Dirk}.

Our Theorem \ref{thm:Lp-L2} brings the probabilistic picture for simplex slicing to the same level now that has been known for Ball's cube slicing since the work \cite{CKT} (interestingly, there is also a single phase transition, occurring already for negative moments). 

Finally, our main result, Theorem \ref{thm:Lp-L1}, falls into the realm of reverse H\"older-type inequalities, a.k.a. Khinchin-type inequalities. H\"older's inequality asserts that $\|X\|_p \leq \|X\|_q$  for an arbitrary random variable $X$ as long as $p \leq q$. Khinchin in his work on the law of the iterated logarithm \cite{Khin} established that for Rademacher sums $X$, such inequalities hold in reverse as well, up to a multiplicative constant depending only on $p$ and $q$. It has been of interest to find other classes of random variables for which the reversals hold and to determine best constants in such inequalities (we refer to the introduction in \cite{HT} for many further references, as well as to \cite{BMNO} for the latest developments).

Log-concave random variables constitute a prominent class enjoying Khinchin-type inequalities, with natural connections to convex geometry (marginals of uniform measures on convex bodies are log-concave), which can be traced back at least to the {\red seminal} paper \cite{MP} by Milman and Pajor; we also refer to Chapter 5 in the lecture notes \cite{GNT}. Even preceding this development (arising in a context of distributions with monotone hazard rate significant in statistics with applications to renewal processes), notable is the result of \cite{BMP} where \eqref{eq:Lp-Lq} was established for \emph{all} $-1 < p < q$ for random variables $X$ which are symmetric or nonnegative (a.s.) with log-concave tails, i.e. when $t \mapsto -\log\p{X > t}$ is convex (implied when $X$ is log-concave). For Khinchin-type inequalities like \eqref{eq:Lp-Lq} with \emph{sharp} constants, weakening the symmetry assumption however has proved challenging and has been investigated only recently (in the case of actual $L_p$ norms). Given $1 \leq p \leq q$, let $C_{p,q}^*$ be the best constant $C$ such that $\|X\|_q \leq C\|X\|_p$ holds for all log-concave random variables with mean $0$. Eitan in \cite{Eitan} (Theorem 3.1 therein) {\red has shown} that $C_{p,q}^* = \max_{0 \leq s \leq 1}\frac{\|Z_s\|_q}{\|Z_s\|_p}$, that is $C_{p,q}^*$ is attained within the one parameter family of two-sided exponential random variables $Z_s = s(\cE-1) - (1-s)(\cE'-1)$, $0 \leq s \leq 1$, $\cE, \cE'$ being i.i.d. standard exponential random variables. Eitan has also conjectured that when $p, q$ are even integers, the maximum is attained at $s = 0, 1$ (a one-sided exponential), and confirmed this for $p=2$, and arbitrary even $q$, as well as for all even integers $p \leq q \leq 100$. Nota bene, he has provided nice applications to geometry of convex bodies characterising simplices as bodies with \emph{heaviest tails}. Murawski in \cite{Mur} has established a concrete clean bound $C_{p,q}^* \leq \frac{p}{q}$. 

Our main result completely determines $C_{1,q}^*$ (as well as $C_{p,1}^*$ when $-1 < p \leq 1$), revealing a phase transition of the extremising distribution. 

It is also worth mentioning a standard application to maximal volume central sections of isotropic convex bodies: the limiting case $p \searrow -1$ of Theorem \ref{thm:Lp-L1} recovers Fradelizi's result \cite{Fra} (Corollary 1 therein), which is in the spirit of a main result of \cite{MTW} (Corollary 4 therein establishes a sharp lower bound on noncentral sections).

\section{Proofs}

\subsection{Overall strategy and the main idea: $L_1$ proxy}
The proof of Theorem \ref{thm:Lp-L1} will be done in two steps.

\emph{Step I:} A reduction to two-sided exponentials.

\emph{Step II:} An optimisation over the one parameter family of two-sided exponentials.

At a high level, Theorem \ref{thm:Lp-L2} amounts to solving the following optimisation problem
\begin{align*}
\inf\big/\sup \int_{\R} |x|^pf(x) \dd x &\quad \text{subject to}\\
 &\quad \int_{\R} f(x) \dd x = 1, \int_{\R} xf(x) \dd x = 0, \int_{\R} x^2 f(x) \dd x = 1,
 \end{align*}
where the $\inf\big/\sup$ (depending on the sign of $p$) is taken over all log-concave functions $f\colon \R \to [0,+\infty)$. The objective functional is in fact linear in $f$,
 so are the constraints. After a compactification of the domain, this naturally lands itself into a problem where the so-called localisation techniques developed
  in \cite{FG} may be applied. {\red That is (skipping many details), through the Krein-Milman theorem and the identification of extreme points as log-concave
   functions on a compact interval with at most $3$ degrees of freedom (in the sense of \cite{FG}), 
the above optimisation problem is reduced to the family of compactly supported log-concave functions $f = e^{-\phi}$ with convex potentials $\phi$
 which are piece-wise affine with at most two linear pieces, 
 leading to an explicit but complicated optimisation problem with $4$ parameters (perhaps intractable, without any further nontrivial ideas). See, e.g. \cite{BN, MNT, MNR, Mur} for further details, where this exact scheme has been employed.
 Depending on how one sets the parameters up for the reduced problem, this might amount to, say, given $a, b > 0$, 
 extremising $\int_{-a}^b |x|^pf(x) \dd x$ with $f(x) = c\exp(-\max\{\alpha x + \beta, \alpha' x + \beta'\})$ with the parameters $a, b, c, \alpha, \alpha', \beta, \beta'$
  constrained by $\int f = 1, \int xf(x) = 0$, $\int x^2f(x) = 1$, leaving out $4$ free parameters.}

To overcome such issues, our key idea is to introduce a \emph{proxy} constraint: instead of the $L_2$ constraint, we impose the $L_1$ constraint
\[ 
\int_{\R} |x|f(x) \dd x = 1.
 \]
Crucially, the mean $0$ constraint  $\int_\R xf(x)\dd x = 0$ is plainly equivalent to
\[ 
\int_0^{+\infty} xf(x) \dd x= \int_{-\infty}^0 (-x)f(x) \dd x,
 \]
thus the $L_1$ constraint in fact asserts that both integrals 
\[ 
\int_0^{+\infty} xf(x) \dd x, \quad \int_{-\infty}^0 (-x)f(x) \dd x
 \]
are fixed to be $\frac12$. Together with the convexity/concavity of $x \mapsto |x|^p$ on $(-\infty, 0)$ and $(0, +\infty)$, this is then leveraged in Step I by a simple optimisation argument on each of the half-lines $(-\infty, 0)$ and $(0, +\infty)$ (without the need of compactification), reducing the whole problem to densities $f = e^{-\phi}$ with piece-wise linear potentials $\phi$ having only two slopes, resulting in an explicit {\red optimisation} problem over only $1$ parameter, then tackled in Step II. 

\subsection{Crossing arguments.}
The workhorse of the main reduction in Step I is a simple but quite powerful method of proving integral inequalities by analysing the number of sign changes of the integrand. It will also be instrumental for many tweaks in Step II to keep painstaking technicalities to a minimum. We shall illustrate this technique now with a toy example (needed later anyway). 
Let $0<x_1<x_2<\dots<x_k$. We say that function $f\colon (0,+\infty) \to \R$ has \emph{sign changes at $x_1,x_2,\dots,x_k$ with sign pattern $(+,-,\cdots)$ (resp. $(-,+,\cdots)$)} if $f$ is positive (resp. negative) on $(0,x_1)\cup (x_2,x_3)\cup \cdots$ and $f$ is negative (resp. positive) on $(x_1,x_2)\cup(x_3,x_4)\cup\cdots$. Similarly, we say that $f$ has \emph{weak sign changes at $x_1,x_2,\dots,x_k$ with sign pattern $(0+,0-,\cdots)$ (resp. $(0-,0+,\cdots)$)} if $f$ is nonnegative (resp. nonpositive) on $(0,x_1)\cup (x_2,x_3)\cup \cdots$ and $f$ is nonpositive (resp. nonnegative) on $(x_1,x_2)\cup(x_3,x_4)\cup\cdots$. If there exist $0<x_1<x_2<\dots<x_k$ such that $f$ has (weak) sign changes at $x_1,x_2,\dots,x_k$, we say that $f$ \emph{admits a finite (weak) sign change pattern}.

\begin{remark}
 It is not true that for every smooth function $f\colon (0,+\infty) \to \R$ and every finite interval $[a,b] \subset (0,+\infty)$, if $f$ takes both positive and negative values on $[a,b]$ then $f\big|_{[a,b]}$ admits a finite weak sign change pattern. However, it becomes true if ``smooth'' is replaced with ``analytic''  or ``convex''.
\end{remark}

\begin{lemma}\label{lm:simple-crossings}
Let $h \colon (0, +\infty) \to \R$ be an integrable function. Suppose $\int_0^{+\infty} h(x) \dd x = 0$. If  $h$ has exactly one weak sign change with sign pattern $(0-,0+)$, and if $h$ is nonzero on a set of positive measure, then
\[ 
\int_0^{+\infty} h(x)x \dd x > 0
 \]
(provided the integrability of $h(x)x$). 
\end{lemma}

\begin{proof}
By the constraint that $h$ integrates to $0$, we have for \emph{every} $\alpha$,
\[ 
\int_0^{+\infty} h(x)x \dd x =  \int_0^{+\infty} h(x)(x-\alpha) \dd x.
 \]
If $h$ changes sign exactly once at say $x_0$, then choosing $\alpha = x_0$, the integrand becomes nonnegative on $(0,x_0)\cup(x_0,+\infty)$. Since it is not true that $h$ is zero almost everywhere, the integral is strictly positive, as desired.
\end{proof}

This concept appears naturally in moment problems, was advanced in \cite{ENT2}, and has been extensively used in a variety of contexts, see \cite{BNZ, BN, CET, Eitan, HNT, mars, melbourne2023discrete, melbourne2023transport, MTY}. Robust tools of this sort can be traced back to the classical work \cite{KN} of Karlin and Novikoff (with a slightly different focus on characterising certain cones of functions and their duals).

\subsection{A reduction to two-sided exponentials.}

Let $\cE, \cE'$ be i.i.d. standard exponential random variables. For two parameters $a, b \geq 0$, let
\begin{equation}\label{eq:def-Xab}
X_{a,b} = a(\cE-1) - b(\cE'-1).
 \end{equation}
Note that the density $g_{a,b}$ of $X_{a,b}$ reads
\begin{equation}\label{eq:gab}
g_{a,b}(x) = \begin{cases}\frac{1}{a+b}e^{-(x+a-b)/a}, & x \geq -(a-b), \\ \frac{1}{a+b}e^{(x+a-b)/b}, & x < -(a-b) \end{cases}
\end{equation}
(with the convention that for $a= 0$, or $b = 0$, the first, or the second case is empty, respectively, where we put $g_{a,b} \equiv 0$). In words, the distribution of $X_{a,b}$ is a mixture of two one-sided exponentials: $a\cE-a+b$ with probability $\frac{a}{a+b}$ and $-b\cE'-a+b$ with probability $\frac{b}{a+b}$. 

The main reductions are done with the aid of the following key lemma leveraging convexity.

\begin{lemma}\label{lm:reduction}
Let $X$ be a log-concave random variable with mean $0$. There are unique $a, b \geq 0$ such that 
\begin{equation}\label{eq:constraints}
\p{X > 0} = \p{X_{a,b} > 0} \qquad \text{and} \qquad \E|X| = \E|X_{a,b}|,
\end{equation}
where $X_{a,b}$ is defined in \eqref{eq:def-Xab}.
Moreover, for a function $\psi\colon \R \to \R$, convex on $(0, + \infty)$, zero outside, we have
\begin{equation}\label{eq:sup-two-slopes}
\E\left[ \psi\left(\frac{X}{\E|X|}\right)\right] \leq  \E\left[ \psi\left(\frac{X_{a,b}}{\E|X_{a,b}|}\right)\right].
\end{equation}
The same inequality holds if $\psi$ is convex on $(-\infty, 0)$ and zero outside. 
\end{lemma}

\begin{proof}
Let $\alpha = \p{X > 0}$. Without loss of generality, we can assume that $\alpha \leq \frac12$ (otherwise, we consider $-X$ instead of $X$). Since $\E X = 0$, by Gr\"unbaum's theorem (\cite[Lemma~2.2.6]{BGVV}), in fact, $\alpha \in [\frac1e, \frac12]$. 
Suppose that $a \geq b$. Then,
\[
\p{X_{a,b} > 0} = \int_0^{+\infty} g_{a,b} = \frac{1}{a+b}\int_0^{+\infty} e^{-(x+a-b)/a} \dd x = \frac{1}{e}\frac{e^{b/a}}{1+b/a}.
\]
Since the function $u \mapsto \frac{1}{e}\frac{e^u}{1+u}$ is continuous and strictly increasing on $[0,{+\infty})$, it is bijective between $[0,1]$ and $[\frac1e, \frac12]$. As a result, there is a unique value of the ratio $b/a$ for which $\p{X_{a,b} > 0} = \alpha$. Furthermore, the mean of $X_{a,b}$ is $0$, so
\[
\frac12\E|X_{a,b}| = \int_0^{+\infty} xg_{a,b}(x) \dd x = \frac{1}{a+b}\int_0^{+\infty} xe^{-(x+a-b)/a} \dd x = a\p{X_{a,b} > 0}.
\]
Therefore, taking $a = \frac{\E|X|}{2\p{X >0}}$ and $b \leq a$ determined by the value of the ratio $b/a$ prescribed above, we obtain \eqref{eq:constraints}.

\begin{figure}
    \centering
\begin{tikzpicture}[scale=1.3]

\draw[->, line width=1.2pt] (-3,0) -- (3,0);
\draw[->, line width=1.2pt] (0,-1) -- (0,3);

\draw[line width=1pt] 
  (-2,-1) -- (-1,2.7) -- (2.5,-1);
\node at (2.8,-1) {$\log g$};

\draw (-1,-0.1) -- (-1,0.1);
\node at (-1,-0.3) {$m$};

\draw[dotted] 
  plot[smooth] coordinates {
  (-1.9,-1) (-1.8,-0.2) (-1.7,0.5) (-1.6,1.1)
  (-1.5,1.6) (-1.4,2.0) (-1.3,2.3) (-1.2,2.37)
  (-1.1,2.36) (-1.0,2.34) (-0.5,2) (0.8,0.9)
  (1,0.7) (1.2,0.45) (1.3,0.3) (1.4,0.1)
  (1.7,-0.4) (2,-1)
  };
\node at (1.6,-1) {$\log f$};

\fill (-1.82,-0.33) circle (0.04);
\fill (-1.09,2.36) circle (0.04);
\fill (0.3,1.32) circle (0.04);
\fill (1.33,0.23) circle (0.04);

\end{tikzpicture}
    \caption{The intersection pattern of $\log f$ and $\log g$.}
    \label{fig:f-g}
\end{figure}
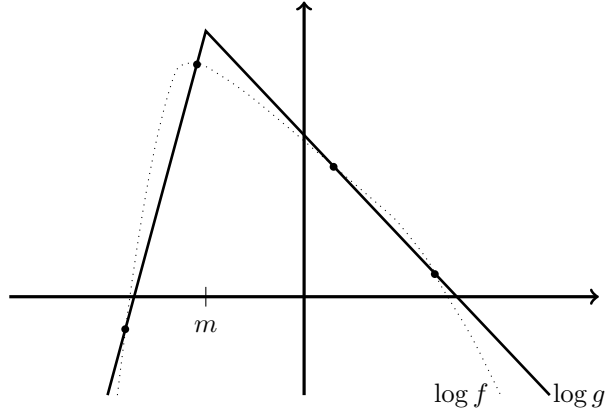

We proceed {\red to show} \eqref{eq:sup-two-slopes}. By the homogeneity of $X/\E|X|$, we can assume that $\E|X| = 1$. Let $f\colon \R \to [0,{+\infty})$ denote the density of $X$. For brevity, we shall write $g = g_{a,b}$ for the density of $X_{a,b}$ given above. These are log-concave probability densities, by construction $\log g$ is piecewise linear with maximum at $m = -(a-b) \leq 0$, enjoying the constraints
\[
\int_0^{+\infty} f = \int_0^{+\infty} g, \qquad \int_0^{+\infty} xf(x) \dd x = \int_0^{+\infty} xg(x) \dd x,
\]
thus also,
\[
\int_{-\infty}^0 f = \int_{-\infty}^0 g, \qquad \int_{-\infty}^0 xf(x) \dd x = \int_{-\infty}^0 xg(x) \dd x
\]
(since $\int f = \int g = 1$ and $\int xf(x) = \int xg(x) = 0$). See Figure \ref{fig:f-g}.

Note that $f$ and $g$ on $[m,{+\infty})$ can intersect \emph{at most} twice, that is $g-f$ changes sign at most twice because $\log f$ is concave and $\log g$ is linear. Similarly on $(-\infty, m]$. Because of the integral constraints on $[0, {+\infty})$, Lemma \ref{lm:simple-crossings} asserts that $g-f$ needs to change its sign \emph{at least} twice on $[0, {+\infty})$, say at $0 \leq x_1 < x_2$, unless $f = g$ (identically). Then, $g - f$ has the sign pattern $(+,-,+)$ on $(0, {+\infty})$. Incidentally (it will be needed in the sequel), we also remark that $g - f \geq 0$ on $[m,0]$ as $\log g$ is linear on the entire half-line $[m,{+\infty})$. We take $\lambda, \mu \in \R$ such that $\psi(x) - \lambda x - \mu$ vanishes at $x_1, x_2$ ($\lambda, \mu$ are thus uniquely determined by the corresponding linear equations). Since $\psi$ is convex, $\psi(x) - \lambda x - \mu$ enjoys the sign pattern $(0+,0-,0+)$. Then,
\begin{align*}
\E \psi(X_{a,b}) - \E\psi(X) &=   \int_0^{+\infty} \psi\cdot (g - f) \\
&= \int_0^{+\infty} \big(\psi(x) - \lambda x - \mu\big)\cdot\big(g(x) - f(x)) \dd x \geq 0,
\end{align*}
as the last integrand is pointwise nonnegative. It remains to address the case when $\psi$ is defined on $(-\infty, 0)$, that is to show that
\[
\int_{-\infty}^0 \psi\cdot (g - f) \geq 0. 
\]
As before, by the integral constraints, $f$ and $g$ must intersect \emph{at least} twice on $(-\infty, 0]$. Since $g - f \geq 0$ on $[m,0]$ and $\log g$ is linear whilst $\log f$ is concave on $(-\infty, m]$, they must intersect \emph{at most} twice on $(-\infty, m]$, thus on $(-\infty,0]$ (since $g-f \geq 0$ on $[m,0]$). Repeating the previous argument verbatim finishes the proof.
\end{proof}

For $0 \leq t \leq 1$, we consider a family 
\[
\cE_t = X_{1,t} = \cE-1 - t(\cE'-1)
\] 
of two-sided exponentials of mean $0$. Letting
\[ 
\mu_t = \| \cE_t  \|_1 = \frac{2}{e}\frac{e^t}{1+t},
 \]
we rescale $\cE_t$ to define a family with the $L_1$ norm fixed to be $1$,
\[ 
\bar\cE_t = \frac{1}{\mu_t}\cE_t.
 \]
Note that $\mu_0 = \frac2e$ and $\mu_1 = 1$ and $t \mapsto \mu_t$ is continuous strictly increasing on $[0,1]$.

Lemma \ref{lm:reduction} readily reduces the proofs of the main inequalities \eqref{eq:Lp-L1-p<1} and \eqref{eq:Lp-L1-p>1} to the family $\{\bar\cE_t\}_{t\in [0,1]}$. Indeed, suppose first that $-1 < p < 0$. Applying \eqref{eq:sup-two-slopes} twice to $\psi(x) = x^p$ on $(0, {+\infty})$ and $\psi(x) = (-x)^p$ on $(-\infty, 0)$, we get
\[ 
\sup \frac{\E|X|^p}{(\E|X|)^p} = \sup_{a, b \geq 0} \frac{\E|X_{a,b}|^p}{(\E|X_{a,b}|)^p}
 \]
with the supremum on the left taken over all mean $0$ log-concave random variables $X$. By the homogeneity and symmetry of the ratio $\frac{\E|X_{a,b}|^p}{(\E|X_{a,b}|)^p}$, we can take $0 \leq b \leq a = 1$ which results in
\[ 
\sup \frac{\E|X|^p}{(\E|X|)^p} = \sup_{0 \leq t \leq 1} \E|\bar\cE_t|^p,
 \]
or, taking the $1/p$ power (which is negative),
\[ 
\inf \frac{\|X\|_p}{\|X\|_1} = \inf_{0 \leq t \leq 1} \|\bar\cE_t\|_p.
 \]
Arguing the same way in the case $0 < p \leq 1$ (using the concavity of $|x|^p$ on both $(0,{+\infty})$ and $(-\infty, 0)$), we conclude that \eqref{eq:Lp-L1-p<1} follows once we prove that
\begin{equation}\label{eq:E_t-p<1}
\inf_{0 \leq t \leq 1} \|\bar\cE_t\|_p = \|\bar\cE_1\|_p =  \Gamma(p+1)^{1/p}, \qquad -1 < p \leq 1.
\end{equation} 
By the same token, \eqref{eq:Lp-L1-p>1} follows once we prove that
\begin{equation}\label{eq:E_t-p>1}
\sup_{0 \leq t \leq 1} \|\bar\cE_t\|_p =  \max\left\{\|\bar\cE_1\|_p, \|\bar\cE_0\|_p\right\} = C_p, \qquad p \geq 1.
\end{equation}

{\red Conceptually, one might therefore think of Lemma \ref{lm:reduction} as a localisation step, in that it reduces the optimisation problem at hand over \emph{all} (constrained) log-concave densities to \emph{simple} concrete densities (two-sided exponentials). As discussed earlier, the off-hand  localisation method of Fradelizi and Guedon from \cite{FG} does not seem to be a feasible approach as the resulting problem involves too many parameters to handle. Finally, as another point of comparison, Eitan's main argument from \cite{Eitan} leads to the same reduction (Theorem 3.1 specialised to $p=1$ therein). His approach also relies on crossing arguments, but via direct comparison, viz. bounding the number of zeros of polynomial-type functions (Chebyshev's systems). Our application of Lemma~\ref{lm:reduction} on the other hand strongly relies on the convexity of $|x|^p$ as well as the interplay between the mean zero assumption and the $L_1$ constraint, which we have strived to make explicit by phrasing Lemma \ref{lm:reduction} in a more robust setting concerning the optimisation of integrals against convex functionals $\psi$.}

{\red It is also worth comparing Lemma \ref{lm:reduction} to Fradelizi's results from \cite{Fra}, in particular Theorem 8, and the upper bound of (6), therein, which states that among all log-concave probability densities $f$ of mean $0$, with a \emph{fixed} value at $0$, the symmetric exponential density maximises the integral against any \emph{even} convex function. We can recover Fradelizi's upper bound using Lemma \ref{lm:reduction} (nota bene, the assumption that $\phi$ is even is not needed). We state this as the following corollary to Lemma \ref{lm:reduction}, and sketch the argument.

\begin{corollary}[Upper bound (6) in \cite{Fra}]\label{cor:Fra}
Let $\phi\colon \R\to \R$ be a convex function. For every log-concave probability density function $f\colon \R \to [0,+\infty)$ with mean $0$, that is $\int_{-\infty}^\infty xf(x) \dd x = 0$, we have
\[ 
\int_{-\infty}^\infty \phi(x)f(x) \dd x \leq \int_{-\infty}^\infty \phi(x)\cdot \Big(f(0)e^{-2f(0)|x|}\Big)\dd x.
 \]
\end{corollary}
\begin{proof}[Sketch]
We apply \eqref{eq:sup-two-slopes} twice, to $\psi(x) = \phi(\lambda x)$, $\lambda = \int |x|f(x)$, restricted to $(0, + \infty)$ and $(-\infty, 0)$, and obtain $\int \phi\cdot g \geq \int \phi\cdot f$, where $g = g_{a,b}$ is the two-sided exponential density \eqref{eq:gab}, matching the quantities $\int_0^\infty f$ and $\int_0^\infty tf(t) \dd t$, (instead of $f(0)$). We then consider $\tilde g(x) = g(0)\exp\{-2g(0)|x|\}$ and argue that $\tilde g - g$ has sign pattern $(+,-,+)$. Since $\phi$ is convex, and both densities $g$ and $\tilde g$ have mean $0$, this gives $\int \phi\cdot \tilde g \geq \int \phi \cdot g$. Now we consider $\tilde f(x) = f(0)\exp\{-2f(0)|x|\}$. Crucially, it can be showed that $g(0) \geq f(0)$ (assuming without loss of generality, $a \geq b > 0$, the calculations in the proof of Lemma \ref{lm:reduction} give $g(0) = \frac{1}{a}\frac{1}{e}\frac{e^{b/a}}{1+b/a} = \frac{(\int_0^\infty f)^2}{\int_0^\infty xf(x)}$ which is at least $f(0)$, e.g. by Lemma 2 in \cite{W}). Thus, equivalently, $\tilde f - \tilde g$ has sign patter $(+,-,+)$, and, as a result, $\int \phi\cdot \tilde f \geq \int \phi\cdot \tilde g$, leading to $\int \phi\cdot \tilde f  \geq  \int \phi\cdot f$, as desired.
\end{proof}
}

Note that $\bar\cE_1 =\cE_1 = \cE -\cE'$ has the standard symmetric double-sided exponential density $\frac12e^{-|x|}$, consequently $|\bar\cE_1|$ has the same distribution as $\cE$, so $\|\bar\cE_1\|_p = \Gamma(p+1)^{1/p}$, whereas $\bar\cE_0 = \frac{1}{\mu_0}(\cE-1) = \frac{2}{e}(\cE-1)$, hence the right-most equalities in \eqref{eq:E_t-p<1} and \eqref{eq:E_t-p>1} are as desired.

Before we proceed with proofs of \eqref{eq:E_t-p<1} and \eqref{eq:E_t-p>1}, we establish necessary supplementary results.

\subsection{Auxiliary lemmata}

Throughout the rest of this paper, we follow the notation introduced in the previous subsection, that is $\cE, \cE'$ are i.i.d. $\text{Exp}(1)$, $\cE_t = \cE-1 - t(\cE'-1)$, $\mu_t = \E|\cE_t|$ and $\bar\cE_t = \mu_t^{-1}\cE_t$, for $t \in [0,1]$.

\begin{lemma}\label{lm:3crossings}
Let $f_t\colon[0,{+\infty}) \to [0,{+\infty})$ be the density of $|\bar\cE_t|$. For every $0<t<1$, each of the functions $f_t - f_0$ and $f_1 - f_t$ changes sign exactly $3$ times on $(0, {+\infty})$ and has sign pattern $(+,-,+,-)$.
\end{lemma}
\begin{proof}
In the notation of \eqref{eq:def-Xab}, we have $\cE_t = X_{1,t}$, so recalling \eqref{eq:gab}, $\cE_t$ has density $g_{1,t}$, thus $|\cE_t|$ has density on $(0,{+\infty})$ at $x > 0$ equal to
\begin{align*}
\rho_t(x) = g_{1,t}(x) + g_{1,t}(-x) &= \frac{1}{1+t}e^{-(x+1-t)} +  \frac{1}{1+t}\begin{cases}
e^{x-(1-t)}, & x \leq 1-t, \\ e^{-x/t+(1-t)/t}, & x > 1-t,
\end{cases}\\
&=\frac{e^{t-1}}{1+t}\begin{cases}
e^{x}+e^{-x}, & x \leq 1-t, \\ e^{-x} + e^{-x/t + 1/t-t}, & x > 1-t.
 \end{cases}
 \end{align*}
When $t = 0$ (following the convention in \eqref{eq:gab} that $e^{-x/t+(1-t)/t}$ is replaced by $0$ identically on $(1,{+\infty})$),
\[ 
\rho_0(x) = e^{-1}\begin{cases}
e^{x}+e^{-x}, & x \leq 1, \\ e^{-x}, & x > 1.
 \end{cases}
 \]
Plainly 
\[ 
\rho_1(x) = e^{-x}.
 \]
 Note that for every $0 < t \leq 1$, $\rho_t$ is continuous, whereas $\rho_0$ has a jump at $x = 1$ and is continuous elsewhere.

The density $f_t$ of $|\bar\cE_t| = \mu_t^{-1}|\cE_t|$ is then $f_t(x) = \mu_t\rho_t(\mu_tx)$. 

Fix $0 < t < 1$. First let us handle the function 
\[
h_1(x) = f_1(x) - f_t(x) = \rho_1(x) - \mu_t\rho_t(\mu_tx).
\]
Since $\int_0^{+\infty} h_1 = 0$, and $h_1$ is not identically $0$, it must change sign at least once; additionally $\int_0^{+\infty} h_1(x)x \dd x = 0$, so $h_1$ need change sign at least twice on $(0, {+\infty})$, by Lemma \ref{lm:simple-crossings}.
It is evident that the ratio $r_1(x) = f_t(x)/f_1(x) = \mu_t\rho_t(\mu_t x)e^x$ is convex on each of the intervals $[0, \alpha_t]$ and $[\alpha_t, {+\infty})$, where $\alpha_t = \mu_t^{-1}(1-t)$ denotes the point where the formula defining $f_t(x)$ changes. Moreover, $r_1(x)$ is strictly increasing on $[0, \alpha_t]$ (as the product of two strictly increasing functions) with $r_1(0) = \mu_t\rho_t(0) = \mu_t^2 < 1$. It is also clear that $r_1(x) \to {+\infty}$ as $x \to {+\infty}$. If $r_1(\alpha_t) \leq 1$, then $r_1(x)=1$ would have no solutions on $(0,\alpha_t)$ and exactly $1$ solution on $(\alpha_t,{+\infty})$, resulting in $f_1 - f_t$ having at most one sign change, ruled out earlier. Therefore, $r_1(\alpha_t) > 1$. Consequently, $r_1(x)=1$ has exactly one solution on $(0,\alpha_t)$ (by monotonicity) and exactly $2$ solutions on $(\alpha_t,{+\infty})$ (by convexity), resulting in $f_1 - f_t$ having exactly three sign changes and the sign pattern $(+,-,+,-)$ (because $r_1(0) < 1$).

Now we handle the function
\[
h_0(x) = f_t(x) - f_0(x) = \mu_t\rho_t(\mu_tx) - \mu_0\rho_0(\mu_0x).
\]
As before, it necessarily has at least two sign changes.

We will frequently use that $\mu_t > \mu_0$. Recall $\rho_0(\mu_0x)$ has a jump discontinuity at $x_0 = \mu_0^{-1}$ and the formula for $\rho_t(\mu_tx)$ changes at $x = \alpha_t = \mu_t^{-1}(1-t) < \mu_0^{-1} = x_0$. It will be convenient to split the analysis into the consecutive intervals: $[0, \alpha_t]$, $(\alpha_t, x_0]$ and $(x_0,+ \infty)$.

On $[0, \alpha_t]$, we have $\frac12 h_0(x) = \mu_t\frac{e^{t-1}}{1+t}\cosh(\mu_tx) - \mu_0e^{-1}\cosh(\mu_0x)$. Since $\mu_t > \mu_0$ and $\frac{e^{t-1}}{1+t} > e^{-1}$, we have $h_0 > 0$ on $[0, \alpha_t]$.

On $[\alpha_t, x_0]$, $f_t$ is now strictly decreasing, whereas $f_0(x) = \mu_0^2\cosh(\mu_0 x)$ is still strictly increasing, so $h_0$ changes sign at most once on this interval.

On $(x_0, {+\infty})$, the ratio 
\[
r_0(x) = \frac{f_t(x)}{f_0(x)} = \frac{e\mu_t}{\mu_0}\rho_t(\mu_tx)e^{\mu_0x} = \frac{\mu_t}{\mu_0}\frac{e^t}{1+t}\left(e^{-(\mu_t-\mu_0)x} +e^{-(\mu_t/t-\mu_0)x + 1/t-t} \right)
\]
is clearly strictly decreasing with $r_0(x) \to 0$ as $x \to {+\infty}$. In particular, $h_0$ can have at most one sign change here, and is eventually negative. 

Function $h_0$ admits one more sign change, at the discontinuity point $x_0$, so at most $3$ sign changes on $(0, {+\infty})$, and minimum $2$ sign changes. As remarked, near $0$, $h_0$ is positive, and eventually negative, thus it {\red needs to} change sign an odd number of times. Thus we get exactly $3$ sign changes. Since $f_t(x) < f_0(x)$ as $x \to {+\infty}$, the sign pattern is $(+,-,+,-)$.

\end{proof}

\begin{lemma}\label{lm:Lp-explicit}
For $0 \leq t \leq 1$ and $p > -1$, we have
\[ 
\E|\cE_t|^p = \frac{e^{t-1}}{1+t}\left(\int_0^{1-t} x^{p}e^x \dd x + \Gamma(p+1)\right) +  \frac{t}{1+t}\E(t\cE + 1-t)^{p}.
 \]
\end{lemma}
\begin{proof}
Recall $\cE_t = \cE-t\cE'-(1-t)$. Using that the distribution of $\cE-t\cE'$ is the same as a mixture of two exponentials: $\cE$ with probability $\frac{1}{1+t}$, and $-t\cE$ with probability $\frac{t}{1+t}$, we have
\[ 
\E|\cE_t|^{p} = \frac{1}{1+t}\E|\cE - (1-t)|^{p} + \frac{t}{1+t}\E(t\cE + 1-t)^{p},
 \]
and we calculate the first expectation,
\begin{align*}
\E|\cE - (1-t)|^{p} &= \int_0^{+\infty} |x - (1-t)|^{p} e^{-x} \dd x \\
&= \int_0^{1-t} (1-t-x)^{p}e^{-x}\dd x + \int_{1-t}^{+\infty} (x-(1-t))^{p} e^{-x}\dd x \\
&= e^{t-1}\left(\int_0^{1-t} x^{p}e^x \dd x + \Gamma(p+1)\right).
\end{align*}
\end{proof}

\begin{lemma}\label{lm:low-moments}
{\red The following properties hold:}
\begin{enumerate}[(i)]
\item\label{lm:low-mom-2}
The function $t \mapsto \|\bar\cE_t\|_2$ is strictly increasing on $[0,1]$.

\item\label{lm:low-mom-3}
The function $t \mapsto \|\bar\cE_t\|_3$ is unimodal on $[0,1]$: first strictly decreasing, then strictly increasing, with $\|\bar\cE_t\|_3 \leq  \|\bar\cE_1\|_3$ for all $\frac15 \leq t \leq 1$.

\item\label{lm:low-mom-4}
The function $t \mapsto \|\bar\cE_t\|_4$ is strictly decreasing on $[0,1]$.
\end{enumerate}
\end{lemma}
\begin{proof}[Proof of \eqref{lm:low-mom-2}]
We have,
\[ 
\|\bar\cE_t\|_2^2 = \mu_t^{-2}\Var(\cE_t) = \mu_t^{-2}(1+t^2) = \frac14e^{2(1-t)}(1+t^2)(1+t)^2.
 \]
Thus
\[ 
\frac{\dd}{\dd t}\|\bar\cE_t\|_2^2 = \frac12e^{2(1-t)}t^2(1-t^2) \geq 0,
 \]
with equality only at the endpoints $t = 0,1$, so the claim follows.

\emph{Proof of \eqref{lm:low-mom-3}.} We use Lemma \ref{lm:Lp-explicit},
\[ 
\E|\cE_t|^3 = 2\left(\frac{6e^{t-1}}{1+t} + t^3-1\right),
 \]
and after some calculations, we arrive at
\[ 
\frac{\dd}{\dd t}\|\bar\cE_t\|_3^3 = \frac34e^{3(1-t)}t(1+t)\Big[2t^2+2t+1-t^4 - 4e^{t-1}\Big].
 \]
Its sign is the same as the sign of the function in the square parentheses $h(t) = 2t^2+2t+1-t^4 - 4e^{t-1}$ which we now analyse. Since $h'''(t) = -24t-4e^{t-1} < 0$, $h'$ is concave. Moreover, $h'(0) = 2 - 4/e > 0$ and $h'(1) = -2 < 0$, so $h'$ is first positive then negative on $[0,1]$. Consequently, $h$ first increases then decreases. Since $h(0) = 1 - 4/e < 0$ and $h(1) = 0$, we conclude that $h$ is first negative, then positive on $(0,1)$ with exactly one zero. As a result $t \mapsto\E|\bar\cE_t|^3$ is first decreasing then increasing. Finally, we check that $\E|\bar\cE_{1/5}|^3 = 5.97.. < 6 = \E|\bar\cE_{1}|^3$ which finishes the proof.

\emph{Proof of \eqref{lm:low-mom-4}.} We have,
\[ 
\E|\cE_t|^4 = \E\Big((\cE-1)-t(\cE'-1)\Big)^4 = (1+t^4)\E(\cE-1)^4 + 6t^2\Big(\E(\cE-1)^2\Big)^2 = 9t^4 + 6t^2 + 9.
 \]
Then, using $\frac{\dd}{\dd t}\mu_t = \mu_t\frac{t}{1+t}$, we find that
\[ 
\frac{\dd}{\dd t}\|\bar\cE_t\|_4^4 =-12\frac{t(1-t)^2}{\mu_t^4(1+t)}\Big(3t^2+3t+2\Big)
 \]
which is clearly negative on $(0,1)$, hence the desired function is decreasing.
\end{proof}

As in the proof of Lemma \ref{lm:reduction}, we will repeatedly argue that certain integral inequalities hold by rearranging the integrand to show that it is pointwise nonnegative,
 analysing sign patterns of sums of certain power functions. We record the following elementary lemma, adapted to concrete functions stemming from our setting.  
 {\red Using an exponential change of variables, the following can be found as the $n=3$ case of \cite[Lemma A.2]{Fer}.}

\begin{lemma}\label{lm:power-functions}
Given distinct $p,q \in \R$ such that $q > 2$ and $0, 1, p, q$ are pairwise distinct, and positive numbers $0 < x_1 < x_2 < x_3$, there are unique $\alpha, \beta, \gamma \in \R$ such that the function
\[ 
x \mapsto x^p - (\alpha + \beta x + \gamma x^q)
 \]
changes sign on $(0,{+\infty})$ exactly at $x_1, x_2, x_3$. Moreover, it has the sign pattern
\begin{enumerate}[(i)]
\item\label{lm:power-functions1} $(+,-,+,-)$ when $p \in (-1,0) \cup (1, q)$, 
\item\label{lm:power-functions2} $(-,+,-,+)$ when $p \in (0,1) \cup (q, {+\infty})$.
\end{enumerate}
\end{lemma}
\begin{proof}
It is a theorem that for any reals $a_1<a_2<\dots<a_n$ and $b_1<b_2<\dots<b_n$, the Vandermonde-type matrix $(\exp(a_ib_j))_{1\le i,j \le n}$ has positive determinant (see, e.g. Problem 65 in Part Seven of \cite{PS}).
Thus, the system of linear equations 
\[ 
\Big\{ \alpha + \beta x_j + \gamma x_j^q = x_j^p, \qquad j = 1, 2, 3,
 \]
has a unique solution for $\alpha, \beta, \gamma$ (as its matrix is Vandermonde-type and thus has nonzero determinant). This choice of the parameters $\alpha, \beta, \gamma$ results in the function
\[ 
g(x) = x^p - (\alpha + \beta x + \gamma x^q)
 \]
having zeros at $x_1, x_2, x_3$. 

Now let $x'\in(0,x_1)$ be arbitrary. By column operations, we compute \[\det\left[\begin{smallmatrix}
x'^p&1&x'&x'^q\\ x_1^p & 1& x_1 & x_1^q \\ x_2^p &1& x_2 & x_2^q\\ x_3^p &1& x_3 & x_3^q
\end{smallmatrix}\right] = \det\left[\begin{smallmatrix}
g(x')&1&x'&x'^q\\ g(x_1) & 1& x_1 & x_1^q \\ g(x_2) &1& x_2 & x_2^q\\ g(x_3) &1& x_3 & x_3^q
\end{smallmatrix}\right] = \det\left[\begin{smallmatrix}
g(x')&1&x'&x'^q\\ 0 & 1& x_1 & x_1^q \\ 0 &1& x_2 & x_2^q\\ 0 &1& x_3 & x_3^q
\end{smallmatrix}\right] = g(x') \det\left[\begin{smallmatrix}
  1& x_1 & x_1^q \\ 1& x_2 & x_2^q\\ 1& x_3 & x_3^q
\end{smallmatrix}\right].\] Using the theorem about Vandermonde-type determinants, we see \begin{align*}
    (0-p)(1-p)(q-p)\det\left[\begin{smallmatrix}
x'^p&1&x'&x'^q\\ x_1^p & 1& x_1 & x_1^q \\ x_2^p &1& x_2 & x_2^q\\ x_3^p &1& x_3 & x_3^q
\end{smallmatrix}\right] &> 0 \\
\det\left[\begin{smallmatrix}
  1& x_1 & x_1^q \\ 1& x_2 & x_2^q\\ 1& x_3 & x_3^q
\end{smallmatrix}\right] &> 0
\end{align*} and hence $g(x')$ is nonzero and has the same sign as $(0-p)(1-p)(q-p)$. This holds for any $x' \in (0,x_1)$. Similar reasoning in the cases when $x' \in (x_1,x_2)$ or $(x_2,x_3)$ or $(x_3,+\infty)$ yields the desired result.\end{proof}


\subsection{Proof of \eqref{eq:E_t-p<1}}
Fix $0 < t < 1$. First note that thanks to Lemma \ref{lm:low-moments} \eqref{lm:low-mom-2} and \eqref{lm:low-mom-4}, the function
\[ 
q \mapsto \E|\bar\cE_1|^q - \E|\bar\cE_t|^q
 \]
is positive at $q=2$ and negative at $q=4$, so by its continuity, there is $q=q(t) \in (2,4)$ where it vanishes. We fix such a value and call it $q$. As in Lemma \ref{lm:3crossings}, let $f_t$ be the density of $|\bar\cE_t|$. Fix nonzero $p \in (-1,1)$. Observe that for arbitrary $\alpha, \beta, \gamma \in \R$, we can write
\begin{align*}
\E(|\bar\cE_1|^p - |\bar\cE_t|^p) &= \int_0^{+\infty} \Big(f_1(x) - f_t(x)\Big)x^p \dd x\\
&= \int_0^{+\infty} \Big(f_1(x) - f_t(x)\Big)\cdot\Big(x^p  - (\alpha + \beta x + \gamma x^q)\Big)\dd x
 \end{align*}
as each of the integrals against $\alpha$, $\beta x$ and $\gamma x^q$ is zero. We know from Lemma \ref{lm:3crossings} that the first parenthesis has exactly $3$ sign changes on $(0,{+\infty})$, say at $x_1 < x_2 < x_3$, with sign pattern $(+,-,+,-)$. We choose $\alpha, \beta, \gamma$ from Lemma \ref{lm:power-functions} for these nodes. Moreover, from the lemma, the sign pattern of the second parenthesis is $(+,-,+,-)$ when $-1 < p < 0$, in which case the integrand is pointwise nonnegative and we get
\[ 
\E(|\bar\cE_1|^p - |\bar\cE_t|^p) \geq 0,
 \]
that is $\|\bar\cE_t\|_p \geq \|\bar\cE_1\|_p$ ($p$ is negative). Similarly, when $0 < p < 1$, the sign pattern is flipped resulting with the reverse inequality above and consequently again $\|\bar\cE_t\|_p \geq \|\bar\cE_1\|_p$, as desired. \qed

\subsection{Proof of \eqref{eq:E_t-p>1}}
The argument is broken into two parts. Recall from \eqref{eq:Cp} that $p_0$ is defined as the unique solution to the equation $\|\bar\cE_1\|_p = \|\bar\cE_0\|_p$ in $p$ on $(1,{+\infty})$ (see also Lemma \ref{lm:Cp} in the appendix).

\subsubsection{The whole range $p \geq 1$ reduces to $p = p_0$.}
Suppose \eqref{eq:E_t-p>1} holds for $p=p_0$. We shall show now that it then holds for all $p \geq 1$. 

\emph{Case 1: $1 \leq p \leq p_0$.} Fix $0 < t  < 1$. We proceed exactly as in the proof of \eqref{eq:E_t-p<1}. Recall that the function $q \mapsto \E|\bar\cE_1|^q - \E|\bar\cE_t|^q$ is negative at $q = 4$ and by the assumption made here, is nonnegative at $q = p_0$. We fix then a value of $q = q(t) \in [p_0, 4)$ where it vanishes and invoke Lemmas \ref{lm:3crossings} and \ref{lm:power-functions} \eqref{lm:power-functions1} to conclude that
\[ 
\E(|\bar\cE_1|^p - |\bar\cE_t|^p) = \int_0^{+\infty} \Big(f_1(x) - f_t(x)\Big)\cdot\Big(x^p  - (\alpha + \beta x + \gamma x^q)\Big)\dd x \geq 0
 \]
as in the proof of \eqref{eq:E_t-p<1}.

\emph{Case 2: $p \geq p_0$.} Quite similarly, by fixing $0 < t < 1$, consider $q \mapsto \E|\bar\cE_0|^q - \E|\bar\cE_t|^q$ which is negative at $q = 2$ (Lemma \ref{lm:low-moments} \eqref{lm:low-mom-2}) and nonnegative at $q = p_0$. Thus we can take $q = q(t) \in (2, p_0]$ where the function vanishes. Then, we write
\[ 
\E(|\bar\cE_0|^p - |\bar\cE_t|^p) = \int_0^{+\infty} \Big(f_0(x) - f_t(x)\Big)\cdot\Big(x^p  - (\alpha + \beta x + \gamma x^q)\Big)\dd x \geq 0,
 \]
where in view of Lemmas \ref{lm:3crossings} and \ref{lm:power-functions} \eqref{lm:power-functions2}, the parentheses have the same sign pattern $(-,+,-,+)$, thus $\E|\bar\cE_0|^p \geq\E |\bar\cE_t|^p$, as desired. \qed

\subsubsection{Proof of \eqref{eq:E_t-p>1} at transition point $p=p_0$}

We will use the explicit expressions for the ``low'' moments (2nd, 3rd and 4th) as well as several pointwise estimates. Recall that by the definition of $p_0$, $\|\bar\cE_0\|_{p_0} = \|\bar\cE_1\|_{p_0} = \Gamma(p_0+1)^{1/p_0}$ and that $p_0 = 2.9414..$ Our goal here is to show that 
\[ 
\E|\bar\cE_t|^{p_0} \leq \Gamma(p_0+1), \qquad 0 \leq t \leq 1.
 \]
First we show that this holds true for \emph{all} $t$ large enough, in fact for a range of $p$.

\begin{lemma}\label{lm:large-t}
For every $1 \leq p \leq 3$ and $\frac15 \leq t \leq 1$, we have
\[ 
\|\bar\cE_t\|_{p} \leq \|\bar\cE_1\|_p.
 \]
\end{lemma}
\begin{proof}
Fix $1 \leq p \leq 3$ and $\frac15 < t < 1$. Using Lemma \ref{lm:low-moments} \eqref{lm:low-mom-3} and \eqref{lm:low-mom-4}, the function
\[ 
q \mapsto \E|\bar\cE_1|^q - \E|\bar\cE_t|^q
 \]
is positive at $q = 3$ and negative at $q = 4$, so there is $3 < q < 4$ (depending on $t$) where it vanishes. 
It remains to write
\[ 
\E(|\bar\cE_1|^p - |\bar\cE_t|^p) = \int_0^{+\infty} \Big(f_1(x) - f_t(x)\Big)\cdot\Big(x^p  - (\alpha + \beta x + \gamma x^q)\Big)\dd x 
 \]
and analyse sign patterns using Lemmas \ref{lm:3crossings} and \ref{lm:power-functions} \eqref{lm:power-functions1} to conclude that the integrand is pointwise nonnegative.
\end{proof}

Now we finish the whole proof handling small $t$ with a bit of technical work.

\begin{lemma}\label{lm:small-t}
For every $0 \leq t \leq \frac15$, we have
\[ 
\|\bar\cE_t\|_{p_0} \leq \|\bar\cE_1\|_{p_0}.
 \]
\end{lemma}
\begin{proof}
Plainly, the assertion is equivalent to
\[ 
\E|\cE_t|^{p_0} \leq \mu_t^{p_0}\Gamma(p_0+1), \qquad 0 \leq t \leq \frac15.
 \]
Invoking Lemma \ref{lm:Lp-explicit},
\[ 
\E|\cE_t|^{p_0} = \frac{e^{t-1}}{1+t}\left(\int_0^{1-t} x^{p_0}e^x \dd x + \Gamma(p_0+1)\right) +  \frac{t}{1+t}\E(t\cE + 1-t)^{p_0}.
 \]
Since $p_0 < 3$, we can estimate the expectation using the $3$rd moment as follows,
\[ 
\E(t\cE + 1-t)^{p_0} \leq \left(\E(t\cE + 1-t)^3\right)^{p_0/3} = (2t^3+3t^2+1)^{p_0/3} \leq 2t^3 + 3t^2+1.
 \]
Note that there is equality at $t = 0$. Putting these together, it suffices to show that
\[ 
\frac{e^{t-1}}{1+t}\left(\int_0^{1-t} x^{p_0}e^x \dd x + \Gamma(p_0+1)\right) + \frac{t}{1+t}(2t^3 +3t^2+1) \leq \mu_t^{p_0}\Gamma(p_0+1),
 \]
for $0 \leq t \leq \frac15$. Recall $\mu_t = \frac{2e^{t-1}}{1+t}$, so the last inequality is in turn equivalent to
\[ 
\int_0^{1-t} x^{p_0}e^x \dd x + \Gamma(p_0+1) + e^{1-t}t(2t^3 + 3t^2+1) \leq 2^{p_0}\Gamma(p_0+1)e^{(t-1)(p_0-1)}(1+t)^{1-p_0}.
 \]
Since we have equality at $t=0$, it suffices to show that the derivative of the difference between the right and left hand sides is nonnegative for $0 \leq t \leq \frac15$, that is
\[ 
-(1-t)^{p_0}e^{1-t} + e^{1-t}(5t^3+9t^2+1-2t^4-t) \leq 2^{p_0}\Gamma(p_0+1)(p_0-1)e^{(t-1)(p_0-1)}t(1+t)^{-p_0},
 \]
or, after diving both sides by $e^{1-t}$,
\[ 
-(1-t)^{p_0} + 5t^3+9t^2+1-2t^4-t \leq A_{p_0}te^{tp_0}(1+t)^{-p_0},
 \]
where we have set
\[
A_{p_0} = 2^{p_0}e^{-p_0}\Gamma(p_0+1)(p_0-1) = 4.39...
 \]
To show the last inequality, note that for $0 \leq t \leq \frac15$, crudely, 
\begin{align*}
-(1-t)^{p_0} + 5t^3+9t^2+1-2t^4-t  &\leq -(1-t)^{3} + 5t^3+9t^2+1-2t^4-t\\
&= 2t(3t^2+3t+1-t^3)\\
&\leq 2t(3t^2+3t+1) \leq 4t,
\end{align*}
whereas $e^{tp_0}(1+t)^{-p_0} \geq 1$, which finishes the proof.
\end{proof}

\section{Final remarks}

\subsection{Variance constraint.} 
{\red As has emerged} from our motivating discussion, it is perhaps most natural to consider extremising the $L_p$ norm under the $L_2$ constraint, and Theorem \ref{thm:Lp-L2} leaves open the following question of interest: given $1 < p < 2$ (resp. $p > 2$), what is
\[ 
\inf \text{    (resp.} \ \sup) \quad  \frac{\|X\|_p}{\|X\|_2}
 \]
over all mean $0$ log-concave random variables $X$? 

As pointed out earlier, Eitan's Theorem 3.1 from \cite{Eitan} reduces this to the one-parameter family of two-sided exponentials. Based on numerical experiments, we conjecture that the extremising distribution is the symmetric exponential for $1 < p < p^* \approx 1.68$ and one-sided exponential for $p > p^*$, so as for the $L_p-L_1$ problem, with the phase transition point $p^*$ {\red occurring} now in $(1,2)$. This is supported by the case of even $p$ which has been {\red shown} by Eitan (Theorem 1.4 in \cite{Eitan}). If we drop the \emph{mean $0$} constraint, insightful results have been recently obtained by Murawski in \cite{Mur}.

\subsection{Forward H\"older inequalities.}
It would also be of interest to determine sharp constants in usual H\"older inequalities for centred log-concave random variables $X$, that is, most generally, given $-1 < p < q$, what is $ \sup \frac{\|X\|_p}{\|X\|_q}$? 
{\red Although we do not have any compelling conjectures to offer, based on (rather limited) numerical evidence, we suspect that the extremiser is either the (one-sided) exponential or uniform distribution (with sharp phase transitions occuring for small values of $p$ and $q$, close to $-1$, determined by regions defined by nontrivial transcendental equations); in particular, when the competing effect of the unbounded support of the exponential distribution does not play any role, say for all $0 < p < q$, we suspect the extremiser to be simply the uniform distribution. This is to some extent heuristically supported by the fact that in the class of (even) unimodal distributions, for all $-1 < p < q$, the ratio $\frac{\|X\|_p}{\|X\|_q}$ is indeed maximised by the uniform distribution, see Remark 16 in \cite{ENT2}.}

\subsection{Sums of exponentials.}
Khinchin-type inequalities with sharp constants usually concern weighted sums of i.i.d. random variables with quite specific distributions (say Rademacher, Steinhauss, uniform, spherically symmetric, etc.). Our main results make a point that weighted sums of i.i.d. mean $0$ exponentials turn out to be amenable to the natural relaxation to the log-concave setting. In the context of the previous remark, it is perhaps compelling to ask for the sharp constants in  the forward H\"older inequalities for such sums. For instance, given $-1 < p < 2$ {\red (resp. $p > 2$)}, what is
\[ 
{\red \sup} \left\|\textstyle\sum_{j=1}^n a_j(\cE_j -1)\right\|_p \qquad {\red \text{(resp.} \ \inf)}
 \]
subject to $\sum a_j^2 = 1$, going towards a probabilistic extension of Tang's result from \cite{Tang} which gives a tight bound in the limit $p \searrow -1$?

{\red This question is studied in a forthcoming paper \cite{BTT}; we refer to it for precise conjectures and some (sharp) answers. In particular, for every $p > 2$, the infimum above is attained (asymptotically) by a Gaussian extremiser $a_1 = \dots = a_n = n^{-1/2}$, $n \to \infty$; however, in light of Tang's result, we suspect another sharp phase transition as $p$ gets close to $-1$.}

\appendix

\section*{Appendix}\label{appendix}

\begin{proof}[Proof of \eqref{eq:vol-formula}]
We let $f(x) = e^{-\sum_{j=1}^{n+1} x_j}\1_{(0,{+\infty})^{n+1}}(x)$ be the density of the random vector $(\cE_1, \dots, \cE_{n+1})$. For a unit vector $a$, the density of the marginal $\sum_{j=1}^{n+1} a_j\cE_j$ at $t \in \R$ is $\int_{a^\perp + ta} f(x) \dd\vol(x)$ (the integral understood with respect to the $n$-dimensional Hausdorff measure on the affine subspace $a^\perp + ta$). Thus
\[ 
f_{\sum a_j\cE_j}(0) = \int_{a^\perp} f(x) \text{dvol}_{a^\perp}(x).
 \]
On the other hand, using Fubini's theorem with respect to the parallel simplices $H_t = \{x \in \R_+^{n+1}, \ \sum x_j = t\}$, $t \geq 0$, on which $f(x)$ is constant equal to $e^{-t}$, we obtain
\begin{align*}
 \int_{a^\perp} f(x) \text{dvol}_{a^\perp}(x) &= \int_0^{+\infty} \int_{H_t \cap a^\perp} e^{-t}\text{dvol}_{a^\perp}(x) \\
&= \int_0^{+\infty} e^{-t}\vol_{n-1}(H_t \cap a^\perp) \frac{\dd t}{\sqrt{n+1}},
 \end{align*}
where the factor $\frac{1}{\sqrt{n+1}}$ results from the Jacobian. By the homogeneity of volume, $\vol_{n-1}(H_t \cap a^\perp)  = t^{n-1}\vol_{n-1}(H_1 \cap a^\perp)$. Noting that $H_1$ is the simplex $\Delta_n$ and evaluating $\int_0^{+\infty} e^{-t} t^{n-1} \dd t = \Gamma(n)$ finishes the proof.
\end{proof}

\begin{lemma}\label{lm:Cp}
Let $\cE$ be a standard exponential (mean $1$) random variable. The function 
\[ 
h(p) = \Gamma(p+1) - \left(\frac{e}{2}\right)^p\E|\cE-1|^p
 \]
has a unique zero $p = p_0 = 2.9414..$ in $(1, {+\infty})$. Moreover, $h(p) > 0$ on $(1, p_0)$ and $h(p) < 0$ on $(p_0,+ \infty)$.
\end{lemma}
\begin{proof}
Applying Lemma \ref{lm:Lp-explicit} to $t = 0$ yields $\E|\cE-1|^p = e^{-1}\left(\int_0^1x^pe^x \dd x  +\Gamma(p+1)\right)$. It is then checked numerically that $h(2.9414) > 10^{-5}$ and $h(2.9415) < -10^{-5}$, thus $h(p)=0$ has a root $p_0$ in $(2.9414, 2.9415)$. Moreover, $h(1) = 0$.
It remains to argue that $h$ has no more roots on $[1,{+\infty})$ with the desired sign pattern. To this end, we use Lemma \ref{lm:3crossings} (and its notation, namely that $f_1$ is the density $e^{-x}$ on $(0,{+\infty})$ and $f_0$ is the density of $\frac{e}{2}|\cE-1|$) to obtain
\[ 
h(p) = \int_0^{+\infty} x^p(f_1(x) - f_0(x)){\red \dd x,}
 \]
with $f_1 - f_0$ changing sign exactly $3$ times with the sign pattern $(+,-,+,-)$.
To finish, we employ a crossing argument. We can rewrite the last integral, with arbitrary $\alpha, \beta, \gamma \in \R$, as
\[ 
h(p) = \int_0^{+\infty} \Big(x^p-(\alpha+\beta x + \gamma x^{p_0})\Big)\cdot\Big(f_1(x) - f_0(x)\Big)\dd x.
 \]
We apply Lemma \ref{lm:power-functions} to choose $\alpha, \beta, \gamma$ so that the first parenthesis changes sign exactly at the sign changes of the second one and it {\red has} the sign pattern $(+,-,+,-)$ when $1 < p < p_0$, so $h(p) > 0$ for those $p$, and the reverse holds when $p > p_0$.
\end{proof}

\subsection*{Acknowledgements.}
We are indebted to Silouanos Brazitikos for the useful discussions as well as kind hospitality at the University of Crete. We would like to thank the anonymous referee for their many invaluable comments and suggestions.

JM is supported by CONAHCYT grant CBF2023-2024-3907. MR's research supported in part by an AMS-Simons Travel Grant and in part by NSF Grant DMS-2548742 (formerly NSF DMS-2452384). CT and TT's research supported in part by NSF grant DMS-2246484.
Moreover, this material is, in part, based upon work supported by the National Science Foundation under Grant No. DMS-1929284 while MR and CT were in residence at the Institute for Computational and Experimental Research in Mathematics in Providence, RI, during the Harmonic Analysis and Convexity program December 2024.

\end{document}